\documentclass[12pt,reqno]{amsart}

\usepackage{amssymb,bbm,esint,units,url}
\usepackage{booktabs}
\usepackage[bookmarks=false,unicode,colorlinks,urlcolor=red,citecolor=red,linkcolor=blue]{hyperref}
\usepackage{aliascnt}
\usepackage{doi}
\usepackage{tikz}
\usepackage{pgfplots} 
\usepackage{color}

\usepackage[normalem]{ulem}
\usepackage[margin=1.25in]{geometry}

\newtheorem{thmx}{Theorem}

\newaliascnt{corx}{thmx}
\newtheorem{corx}[corx]{Corollary}
\aliascntresetthe{corx}

\newaliascnt{lemma}{theorem}
\newtheorem{lemma}[lemma]{Lemma}
\aliascntresetthe{lemma}

\newaliascnt{proposition}{theorem}
\newtheorem{proposition}[proposition]{Proposition}
\aliascntresetthe{proposition}

\newaliascnt{corollary}{theorem}

\aliascntresetthe{corollary}

\newaliascnt{conjecture}{theorem}

\aliascntresetthe{conjecture}

\newaliascnt{example}{theorem}

\aliascntresetthe{example}

\makeatletter
\def\tagform@#1{\maketag@@@{\ignorespaces#1\unskip\@@italiccorr}}
\let\orgtheequation\theequation
\def\theequation{(\orgtheequation)}
\makeatother
\def\equationautorefname~{}
\newcommand{\arxiv}[1]{%
 \href{http://front.math.ucdavis.edu/#1}{ArXiv:#1}}

\newcommand{\B}{{\mathbb B}}

\newcommand{\D}{{\mathbb D}}
\newcommand{\e}{\varepsilon}

\newcommand{\R}{{\mathbb R}}
\newcommand{\Rn}{{\mathbb R}^n}

\newcommand{\bo}{{\rm O}}

\newcommand{\sn}{\mathbb{S}^{n-1}}

\newcommand{\ds}{\displaystyle}

\newcommand{\fr}[2]{\frac{\ds #1}{\ds #2}}
\newcommand\soutb{\bgroup\markoverwith{\textcolor{blue}{\rule[.5ex]{2pt}{1pt}}}\ULon}

\newcommand{\Krejcirik}{{Krej\v{c}i\v{r}\'{\i}k}}

\begin{document}

\title{From Neumann to Steklov and beyond, via Robin: the Weinberger way}
%\title{From Neumann to Steklov via Robin: the Weinberger way}
%\title{From Neumann via Robin to Steklov: the Weinberger argument}
%\title{Neumann through Robin to Steklov: the Weinberger argument}
%\title{Neumann to Robin to Steklov: the Weinberger argument}
%\title{A journey with Weinberger from Neumann to Robin to Steklov}
%\title{From Neumann to Steklov via Robin: a journey with Weinberger}
%
\author[]{Pedro Freitas and Richard S. Laugesen}
\address{Departamento de Matem\'atica, Instituto Superior T\'ecnico, Universidade de Lisboa, Av. Rovisco Pais 1,
P-1049-001 Lisboa, Portugal {\rm and}
Grupo de F\'isica M\'atematica, Faculdade de Ci\^encias, Universidade de Lisboa,
Campo Grande, Edif\'icio C6, P-1749-016 Lisboa, Portugal}
\email{psfreitas\@@fc.ul.pt}
\address{Department of Mathematics, University of Illinois, Urbana,
IL 61801, U.S.A.}
\email{Laugesen\@@illinois.edu}
\date{\today}


\begin{abstract}
The second eigenvalue of the Robin Laplacian is shown to be maximal for the ball among domains of fixed volume, for negative values of the Robin parameter $\alpha$ in the regime
connecting the first nontrivial Neumann and Steklov eigenvalues, and even somewhat beyond the Steklov regime. The result is close to optimal, since the ball is not maximal when $\alpha$ is sufficiently large negative, and the problem admits no maximiser when $\alpha$ is positive.
\end{abstract}

\maketitle

\begin{center}
\emph{In memory of Hans Weinberger, and the inspiration he provided.}
\end{center}

\medskip

\section{\bf Introduction and results\label{intro}}

The Robin eigenvalue problem for the Laplace operator on a bounded domain $\Omega$ is
\begin{equation}\label{robinproblem}
\begin{split}
- \Delta u & = \lambda u \ \quad \text{in $\Omega$,} \\
\frac{\partial u}{\partial\nu} + \alpha u & = 0 \qquad \text{on $\partial \Omega$,} 
\end{split}
\end{equation}
where $\alpha$ is a real parameter. The eigenvalues, denoted $\lambda_{k}(\Omega;\alpha)$ for $k=1,2,\dots$, are increasing and continuous as functions of the boundary parameter $\alpha$, and for each fixed $\alpha$ satisfy
\[
\lambda_1(\Omega;\alpha) < \lambda_2(\Omega;\alpha) \leq \lambda_3(\Omega;\alpha) \leq \dots \to \infty . 
\]
This Robin problem models wave motion with an absorbing or radiating boundary ($\alpha<0$ or $\alpha>0$). The Robin spectrum connects the Neumann ($\alpha=0$), Dirichlet ($\alpha\to\infty$) and Steklov ($\lambda=0$) eigenvalue problems, and thus generates a global picture of the spectrum \cite{BFK17}.
In this paper we maximize the second Robin eigenvalue:
\begin{thmx}[$\lambda_2$ is maximal for the ball] \label{mainthm}
If $\Omega$ is a bounded Lipschitz domain in $\Rn, n \geq 2$, and $B$ is a ball of the same volume as $\Omega$, then
\[
\lambda_2(\Omega;\alpha) \leq \lambda_2(B;\alpha) , \qquad \alpha \in \Big[-\frac{n+1}{n}R^{-1},0 \Big] ,
\]
where $R$ is the radius of $B$. Equality holds if and only if $\Omega$ is a ball.
\end{thmx}
The value $\alpha=-R^{-1}$ is significant in that it makes $\lambda_2$ vanish for the ball. Thus the theorem ensures $\Omega$ has at least two negative Robin eigenvalues whenever $\alpha < -R^{-1}$.

From maximality of the ball at the values $\alpha=0$ and $\alpha=-R^{-1}$ we recover maximality of the first nontrivial Neumann and Steklov eigenvalues:
\begin{corx}[Steklov $\sigma_1$ and Neumann $\mu_1$ are maximal for the ball] \label{brockweinberger}
If $\Omega$ is a bounded Lipschitz domain in $\Rn, n \geq 2$, and $B$ is a ball of the same volume as $\Omega$, then
\[
\sigma_1(\Omega) \leq \sigma_1(B) \qquad \text{and} \qquad \mu_1(\Omega) \leq \mu_1(B),
\]
with equality if and only if $\Omega$ is a ball.
\end{corx}
The inequalities for $\mu_1$ and $\sigma_1$ were first proved in the simply connected planar case by Szeg\H{o} \cite{S54} and Weinstock \cite{We54}, respectively, using complex analytic techniques. The results were generalized to arbitrary domains in $n$-dimensions by Weinberger \cite{W56} for $\mu_1$, and by Brock \cite{B01} for $\sigma_1$ (who further obtained maximality of the ball for the harmonic mean of $\sigma_1,\dots,\sigma_n$). In fact, Weinstock normalized the perimeter rather than area of the domain, and so his result on $\sigma_1$ is stronger than Brock's, in $2$ dimensions. Bucur \emph{et al.}\ \cite{BFNT17} recently strengthened the inequality on $\sigma_1$ to surface area normalization in all dimensions, for the class of convex domains. For history and recent developments on Robin, Steklov and Neumann prob{\-}lems, we recommend the open access book on spectral shape optimization edited by Henrot \cite{H17}.

\autoref{brockweinberger} makes explicit a relation in \autoref{mainthm} between the Neumann and Steklov eigenvalue inequalities. These eigenvalues had, until now, been regarded as representing different aspects of the spectral theory of the Laplace operator, probably because they  lie on different axes in the spectral plane: the Neumann eigenvalue is the $\lambda$-intercept of the curve $\alpha \mapsto \lambda_2(\Omega;\alpha)$, while the Steklov eigenvalue is its $\alpha$-intercept.

Our proof of \autoref{mainthm} is inspired by Weinberger \cite{W56}, making use of the Rayleigh quotient 
\begin{equation*}\label{robinrayleigh}
Q[u] = Q[u;\alpha] = 
\frac{\int_\Omega |\nabla u|^2 \, dx + \alpha \int_{\partial \Omega} u^2 \, dS}{\int_\Omega u^2 \, dx} , \qquad u \in H^1(\Omega) .
\end{equation*}
The domain has Lipschitz boundary, and so $H^1(\Omega)$ imbeds compactly into $L^2(\Omega)$. Hence the Robin spectrum is well defined and discrete, and given by the usual minimax variational principles in terms of the Rayleigh quotient. 

The difficulty in the Robin case, when compared to the Neumann case ($\alpha=0$), lies in handling the integral over the boundary in the numerator of the Rayleigh quotient. We start in \autoref{boundary} by estimating the boundary integral with an integral over the domain, which then enables us to apply centre of mass and transplantation arguments as in Weinberger's method. The first part of the paper is dedicated to these preliminaries, and to ascertaining the necessary monotonicity properties for the Robin eigenfunctions of the ball. 
 %---\autoref{boundaryintegralsec} to \autoref{preliminaries}. 
 The theorem and corollary are then proved in \autoref{higherdimproof1} for $-R^{-1} \leq \alpha \leq 0$, with the proof extended to $- \frac{n+1}{n} R^{-1} \leq \alpha < -R^{-1}$ in \autoref{higherdimproof-adapted}. 
 
\subsection*{Extremal domains and conjectures for Robin eigenvalues}
We start by discussing the broader context and literature in extremal spectral geometry for the Robin problem \eqref{robinproblem}. We are interested in the structure of extremal spectral domains under a fixed volume constraint, and in the 
connections to Steklov and Neumann eigenvalues. The nature of the extremal domain can depend in a critical way on the sign of the boundary parameter $\alpha$, which in this paper is assumed to be negative.

\subsubsection*{First eigenvalue} A Faber--Krahn type inequality holds for the first eigenvalue, for each positive $\alpha$, as was proved in two dimensions by Bossel \cite{B86} in 1986, and extended to the $n$-dimensional case by Daners \cite{D06} in 2006, with an alternative approach via the calculus of variations found more recently by Bucur and Giacomini \cite{BG10,BG15a}. 

For negative values of $\alpha$ it was conjectured by Bareket \cite{B77} in 1977 that the ball would now be the global maximiser (not minimiser) among domains of fixed volume. This conjecture appears natural not only because the ball is the extremal domain for the first eigenvalue for most other Laplacian eigenvalue problems, but also due to a perturbation analysis around the Neumann problem ($\alpha=0$). The first Robin eigenvalue curve passes through $(\alpha,\lambda)=(0,0)$ with $\alpha$-derivative equal to $|\partial \Omega|/|\Omega|$, as can be formally seen from the Rayleigh quotient, using that the first eigenfunction is constant when $\alpha=0$. (For more analysis see \cite[Theorem 2.1]{GS07}. Incidentally, that paper also connects  the first Robin eigenvalue to the Ginzburg--Landau theory of superconductivity.) This $\alpha$-derivative is minimal for the ball of the same volume, which leads one to think the ball should have largest first eigenvalue when $\alpha<0$ is small. 

Ferone, Nitsch and Trombetti \cite{FNT15} proved in 2015 that the ball is a local maximiser for the first eigenvalue, when $\alpha<0$, and in the same year Freitas and 
\Krejcirik\ \cite{FK15} showed the disk is a global maximiser among planar domains for each sufficiently small $\alpha<0$.
However, in the latter paper the authors also showed in all dimensions that the ball cannot remain a global maximiser for large (negative) values of the boundary parameter, thus disproving Bareket's conjecture in general. This last result relied on a study of the asymptotic behaviour of
eigenvalues of balls and annular shells as $\alpha \to -\infty$. 

\Krejcirik\  and the first author conjectured that maximisers of the first eigenvalue should still possess radial symmetry whenever $\alpha<0$, and that the global maximiser should switch from a ball to a shell at some critical value of $\alpha$. This conjecture was later supported by numerical evidence \cite[Section 5]{AFK17} showing for planar domains of unit area that an annulus whose radius depends monotonically on $\alpha$ (in a certain fashion) becomes the maximiser for $\alpha < \alpha_{2}^{*}\approx -7.2875$. In three dimensions the transition from the ball to a shell of unit volume is expected to occur at $\alpha_{3}^{*}\approx -6.3757$.\footnote{This value for $\alpha_3^*$ corrects a misprint in \cite{AFK17}.}

For the Bareket conjecture using perimeter normalization instead of area or volume, the disk is the maximiser among planar domains for all $\alpha<0$, by work of Antunes, Freitas and Krej\v{c}i\v{r}\'{\i}k \cite[Theorem 2]{AFK17}, while in higher dimensions the ball is the maximiser among convex domains by Bucur, Ferone, Nitsch and Trombetti \cite{BFNT18}. 

\subsubsection*{Second eigenvalue} The numerical results by Antunes, Freitas and \Krejcirik\ \cite{AFK17} suggest a number of other conjectures. One of these, concerning
the second eigenvalue $\lambda_{2}(\Omega;\alpha)$, was made explicit by Bucur, Kennedy and the first author as Open Problem 4.41 in \cite{BFK17}, and states that the second eigenvalue $\lambda_{2}(\Omega;\alpha)$ should be maximal for the 
ball on a range of values $(\alpha^{*},0)$ for some negative value of $\alpha^{*}$. This conjecture may be seen as a natural
continuation of the Szeg\H{o}--Weinberger maximisation property of the ball for the first nontrivial Neumann eigenvalue \cite{S54,W56}. On the other
hand, and as was pointed out in \cite[Proposition 4.42]{BFK17}, a similar effect to that described above for the first eigenvalue must occur --- the ball cannot remain the global maximiser for all $\alpha<0$. More precisely, and as the numerical results in \cite{AFK17} also
suggest, the value of $\alpha^{*}$ indicated above should correspond to the point where another domain, possibly an annular shell, takes over the role of global
maximiser. The value where the shell and the ball have the same eigenvalue is determined by a somewhat complicated equation involving the modified Bessel functions $I_{m}$ and $K_{m}$; see \cite{FK15}.

\autoref{mainthm} proves this conjecture for the second eigenvalue, on a natural range of $\alpha$ that includes those (negative) values of $\alpha$ for which the second
eigenvalue $\lambda_{2}(B;\alpha)$ of a ball with given volume remains positive. This corresponds to the interval between the Neumann problem
at $\alpha=0$ and the negative of the first nontrivial Steklov eigenvalue of the ball $B$ of radius $R$, which occurs at $\alpha = -1/R$.

\autoref{mainthm} fails when $\alpha>0$. This may be seen by considering a family of rectangles $\mathcal{R}_{L}$ of unit area
with side lengths $L$ and $1/L$. One uses separation of variables and the known bounds on the first eigenvalue of an interval \cite[Appendix A.1]{FK18} (with $a=1/L$ being the short side of the rectangle). These bounds give that for fixed positive $\alpha$,
\[
 \lambda_{1}(\mathcal{R}_{L};\alpha) = 2\alpha L + \bo(1) \qquad \text{as $L \to \infty$,}
\]
so that the first eigenvalue of the rectangle can be arbitrarily large, and hence the second eigenvalue can too. Thus the eigenvalues admit no maximiser, when $\alpha$ is positive.

How far \autoref{mainthm} can continue to hold for values of $\alpha$ below $-\frac{n+1}{n}R^{-1}$ remains to be seen. We expect the result will still
hold for a (bounded) interval of $\alpha$ values below that value. This conjecture is supported by the numerical simulations in \cite{AFK17}.
For domains with unit area ($R=1/\sqrt{\pi}$), the transition between the disk and an annulus having larger second eigenvalue is found in that paper to occur at
$\alpha \approx -6.4050$, while our \autoref{mainthm} is valid for $\alpha \in [-\frac{3}{2}\sqrt{\pi},0] \approx [-2.6589,0]$. In three dimensions ($R=(\frac{3}{4\pi})^{1/3}$
for unit volume), the corresponding transition now occurs at $\alpha \approx -5.5857$, while~\autoref{mainthm} is valid for $\alpha \in [-(4/3)^{4/3}\pi^{1/3},0] \approx [-2.1493,0]$. Just as for the first eigenvalue, these transitions between balls and annular shells are determined by solutions of equations involving the modified Bessel functions $I_{m}, K_{m}$. 

\subsubsection*{Third eigenvalue} The third Robin eigenvalue $\lambda_3(\Omega;\alpha)$  is maximal neither for the ball nor for the double ball of the same volume, when $\alpha<0$, according to numerical work by Antunes \emph{et al.}\ \cite[Figure~4]{AFK17}. That example is surprising, because the Neumann eigenvalue $\mu_2(\Omega)=\lambda_3(\Omega;0)$ is known to be maximal for the double ball, by work of Bucur and Henrot \cite{BH18}. The fact that such Neumann inequalities need not always extend to the Robin 
problem suggests that the validity of \autoref{mainthm} is not obvious \emph{a priori}.

\section{\bf Boundary integral}
\label{boundaryintegralsec}

We need to estimate the boundary integral with a domain integral in the numerator of the Rayleigh quotient. Recall $\Omega$ is a bounded Lipschitz domain. 
\begin{proposition} \label{boundary}
If $f$ is nonnegative and $C^1$-smooth then 
\[
\int_{\partial \Omega} f \, dS 
\geq \int_\Omega \left( \frac{\partial f}{\partial r} + \frac{n-1}{r} f \right) dx
%= \int_\Omega \frac{\partial\ }{\partial r} \left( f r^{n-1} \right) \, dr \, dS(\xi) 
,
\]
%where $r>0$ and $\xi \in \partial \B$ are the standard spherical coordinates; 
and equality holds if $\Omega$ is a ball centered at the origin. Hence 
\begin{equation}\label{boundaryint}
\int_{\partial \Omega} g(r)^2 \, dS \geq \int_\Omega \left( 2g(r)g^\prime(r) + \frac{n-1}{r} g(r)^2 \right) dx 
\end{equation}
whenever $g(r)$ is radial and $C^1$-smooth for $r \geq 0$; equality holds if $\Omega$ is a ball centered at the origin. 
\end{proposition}
\begin{proof}
The radial function $r=|x|$ has slope at most $1$ in each direction, and so 
\begin{align*}
\int_{\partial \Omega} f \, dS
& \geq \int_{\partial \Omega} f \frac{\partial r}{\partial\nu} \, dS && \text{using that $f \geq 0$} \\
& = \int_\Omega ( \nabla f \cdot \nabla r + f \Delta r ) \, dx && \text{by Green's theorem} \\
& = \int_\Omega \left( \frac{\partial f}{\partial r} + \frac{n-1}{r} f \right) dx .
%& = \int_\Omega \frac{\partial\ }{\partial r} \left( f r^{n-1} \right) r^{-n+1} \, dx .
\end{align*}
If $\Omega$ is a ball centered at the origin then $\partial r/\partial\nu=1$ at every boundary point, and so equality holds in the argument above. 

For the final claim of the proposition we want to take $f(x)=g(r)^2$, but this $f$ might not be differentiable at the origin.
So we apply the result on the modified domain $\Omega \setminus \B(\e)$, using that $\int_{\partial \B(\e)} g(r)^2 \, dS \to 0$ as $\e \to 0$ and that $g(r)g^\prime(r)$ and $g(r)^2/r$ are integrable around the origin. 
\end{proof}
The case $f(x)=|x|^2$ of the proposition was used by Brasco, De Philippis and Ruffini \cite[Theorem 7.41]{BDP17} in their quantitative version of Betta, Brock, Mercaldo and Posteraro's weighted isoperimetric inequality, which led them to a quantitative version of Brock's inequality on the first nontrivial Steklov eigenvalue \cite[Theorem 7.44]{BDP17}. Those authors also investigate more general radial weights.

\section{\bf Center of mass}
\label{centerofmass_sec}

A standard center of mass argument will be needed when constructing our trial functions. Let $\Omega$ be a bounded Lipschitz domain, suppose $g(r)$ is continuous  for $r \geq 0$, and define
\[
v_{i+1}(x) = g(r) \frac{x_i}{r} , \qquad i=1,\dots,n .
\]
\begin{proposition}\label{centerofmass}
If $\int_0^\infty g(r) \, dr = \infty$ and $v$ is a nonnegative integrable function with $\int_\Omega v \, dx > 0$, then after a suitable translation of the domain $\Omega$ and the function $v$, the following orthogonality conditions are satisfied:
\[
\int_\Omega v_{i+1} v \, dx = 0 , \qquad i=1,\dots,n .
\]
\end{proposition}
\begin{proof}
Weinberger \cite{W56} proved such a proposition by using Brouwer's fixed point theorem. We follow instead a more direct approach \cite[{\S}7.4.3]{H17} that identifies the desired translation as a minimum point of the Lyapunov function 
\[
L(y) = \int_\Omega F(|y-x|) v(x) \, dx , \qquad y \in \Rn ,
\]
where $F(r)=\int_0^r g(\rho) \, d\rho$. 

The Lyapunov function depends continuously on $y$ by dominated convergence, since $F$ is continuous and $\Omega$ is bounded. Further, $L(y) \to \infty$ as $|y| \to \infty$, because $F(r) \to \infty$ as $r \to \infty$ and the nonnegative function $v$ has positive integral. Hence $L$ achieves a minimum at some point $y$. The partial derivatives at the minimum point must vanish, and so 
\[
0 = \frac{\partial L}{\partial y_i}(y) =  \int_\Omega g(|y-x|) \frac{y_i-x_i}{|y-x|} v(x) \, dx 
\]
for each $i=1,\dots,n$. Changing variable with $x \mapsto x+y$ and writing $|x|=r$ gives
\[
\int_{\Omega-y} g(r) \frac{x_i}{r} v(x+y) \, dx = 0 .
\]
Hence the desired orthogonality holds for $v_{i+1}(x)$ on the translated domain $\Omega-y$, with respect to the translated function $v(\cdot + y)$. 
\end{proof}

\section{\bf Mass transplantation}
\label{mass_sec}

A mass transplantation argument due to Weinberger is used in the proofs. We include the argument for the reader's benefit, and to obtain the ``if and only if'' equality statement. 
\begin{proposition}[Mass transplantation]\label{transplantation}
Suppose $\Omega \subset \Rn$ is a bounded Lipschitz domain having the same volume as the unit ball $\B$. 

If $f(r)$ is decreasing for $r \geq 0$ and is integrable on $\B$, then 
\[
\int_\Omega f(r) \, dx \leq \int_\B f(r) \, dx ,
\] 
and if in addition $f(r)$ is strictly decreasing then equality holds if and only if $\Omega = \B$. 

If $f(r)$ is increasing and is integrable on $\B$, then the inequality reverses direction. 
\end{proposition}
\begin{proof}
Since $f$ is radially decreasing and $\Omega$ and $\B$ have the same measure, we find
\begin{align}
\int_\Omega f(r) \, dx
& = \int_{\Omega \cap \B} f(r) \, dx + \int_{\Omega \setminus \B} f(r) \, dx \notag \\
& \leq \int_{\Omega \cap \B} f(r) \, dx + m(\Omega \setminus \B) f(1) \label{massproof1} \\
& = \int_{\B \cap \Omega} f(r) \, dx + m(\B \setminus \Omega) f(1) \notag \\
& \leq \int_{\B \cap \Omega} f(r) \, dx + \int_{\B \setminus \Omega} f(r) \, dx \label{massproof2} \\
& = \int_\B f(r) \, dx < \infty , \notag
\end{align}
which proves the inequality in the proposition. To prove the equality statement, we will show the inequality in the proposition is strict when $\Omega \neq \B$, assuming $f(r)$ is strictly decreasing. 

The first possibility is that $\Omega \not\subset \B$, so that the open set $\Omega$ contains a point at radius $r \geq 1$ and hence contains a neighborhood outside the unit ball. Thus $m(\Omega \setminus \B)>0$, and since $f$ is strictly decreasing,  inequality \eqref{massproof1} is strict. 

The second possibility is that $\B \not\subset  \Omega$. Then some point of the unit ball lies in the complement of $\Omega$, and near that point lies a neighborhood in $\B \setminus \overline{\Omega}$ (since the boundary of $\Omega$ is locally a Lipschitz graph that separates $\Omega$ from the complement of $\overline{\Omega}$). Thus $m(\B \setminus \Omega)>0$, and since $f$ is strictly decreasing, inequality \eqref{massproof2} is strict. 

If $f$ is increasing then apply the proposition to $-f$ to get the reverse inequality. 
\end{proof}

\section{\bf The Robin spectrum of the ball}
\label{preliminaries}

Consider the Robin eigenvalue problem \eqref{robinproblem} on the unit ball. In spherical coordinates $(r,\theta)\in\R\times\sn$ we may separate variables
in the form 
\[
u(r,\theta)=g(r)T(\theta)
\] 
to obtain that the angular part $T$ satisfies
\[
 \Delta_{\sn}T(\theta) + \kappa(\kappa+n-2) T(\theta) = 0
\]
where $\kappa \geq 0$ is an integer. 
%The first of these equations has eigenvalues with multiplicities given by
%\[
% \dbinom{n+\kappa-1}{n-1} - \dbinom{n+\kappa-3}{n-1}.
%\]
When $\kappa=0$ (giving a constant function $T$) the 
eigenfunctions on the ball are purely radial. For positive values of $\kappa$ the  angular function $T$ is a spherical harmonic, and the eigenvalues have multiplicity greater than $1$.

The radial part $g$ satisfies the Bessel-type equation
\begin{equation}\label{besseleq}
 g''(r) + \fr{n-1}{r}g'(r) + \left( \lambda - \frac{\kappa(\kappa+n-2)}{r^2} \right) g(r) = 0.
\end{equation}

In this section we determine the Robin spectrum of the ball, for every real $\alpha$. For the purposes of the rest of the paper, the key facts about the first and second eigenvalues and eigenfunctions are summarized in the next propositions, and shown graphically in \autoref{Robin_g_first}, \autoref{Robin_g_second} and \autoref{Robin12}. The propositions themselves follow from the remainder of the section. 
\begin{proposition}[First Robin eigenfunction of the ball]\label{basic1} The first eigenvalue is simple and the first eigenfunction is radial ($\kappa=0$), for each $\alpha$.

(i) If $\alpha< 0$ then $\lambda_1(\B;\alpha) < 0$ and the eigenfunction $g(r)$ is positive and radially strictly increasing, with   
\[
g(0) > 0 , \quad g^\prime(0)=0 , \qquad g^\prime(r) > 0 , \quad r \in (0,1) .
\]

(ii) If $\alpha=0$ then $\lambda_1(\B;0) = 0$, with constant eigenfunction $g(r) \equiv 1$.

(iii) If $\alpha > 0$ then $\lambda_1(\B;\alpha) > 0$ and the eigenfunction $g(r)$ is positive and radially strictly decreasing, with   
\[
g(0) > 0 , \quad g^\prime(0)=0 , \qquad g^\prime(r) < 0 , \quad r \in (0,1) .
\]
\end{proposition}
The first eigenfunction is plotted for various values of $\alpha$ in \autoref{Robin_g_first}, for the unit disk in $2$-dimensions. 
\begin{figure}
\includegraphics[scale=0.6]{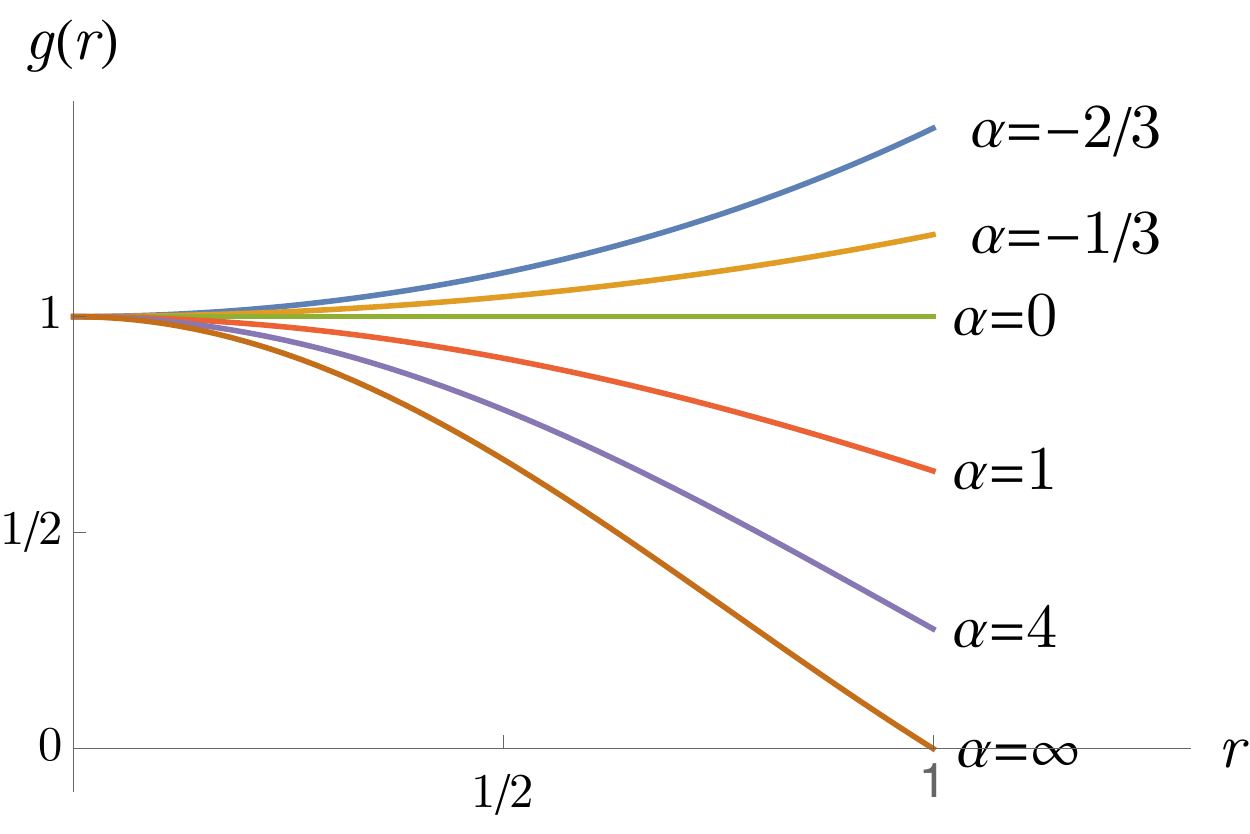}
\caption{\label{Robin_g_first}Plot of the first (radial) Robin eigenfunction of the unit disk, for various values of $\alpha$, normalized with $g(0)=1$. When $\alpha=0$ it is the constant Neumann eigenfunction $1$, and when $\alpha=\infty$ it is the Dirichlet eigenfunction $J_0(j_{0,1}r)$.}
\end{figure}
For the second eigenvalue, recall the spherical harmonics when $\kappa=1$ are the functions $x_1/r,\dots,x_n/r$ (multiplicity $n$). For example, in $2$-dimensions, they are $\cos \theta$ and $\sin \theta$. We call the case $\kappa=1$ ``simple angular dependence''.
\begin{proposition}[Second Robin eigenfunctions of the ball]\label{basic2} The second eigenfunctions have simple angular dependence, meaning they take the form $g(r)x_i/r$ for $i=1,\dots,n$. The radial part $g$ has $g(0)=0,g^\prime(0)>0,g(r)>0$ for $r \in (0,1)$, and when $\alpha \leq 0$ it is strictly increasing, with $g^\prime(r)>0$. 

(i) If $\alpha < -1$ then
\[
\lambda_2(\B;\alpha)=\dots=\lambda_{n+1}(\B;\alpha) < 0,
\]
and $r g^\prime(r) + \alpha g(r) < 0$ for $r \in (0,1)$.

(ii) If $\alpha= -1$ then $g(r)=r$ and 
\[
\lambda_2(\B;-1)=\dots=\lambda_{n+1}(\B;-1) = 0 .
\]

(iii) If $\alpha > -1$ then
\[
\lambda_2(\B;\alpha)=\dots=\lambda_{n+1}(\B;\alpha) > 0,
\]
and $r g^\prime(r) + \alpha g(r) > 0$ for $r \in (0,1)$.
\end{proposition}
The radial part $g(r)$ of the second eigenfunction is plotted for several values of $\alpha$ in \autoref{Robin_g_second}, for the unit disk in $2$-dimensions. 
\begin{figure}
\includegraphics[scale=0.6]{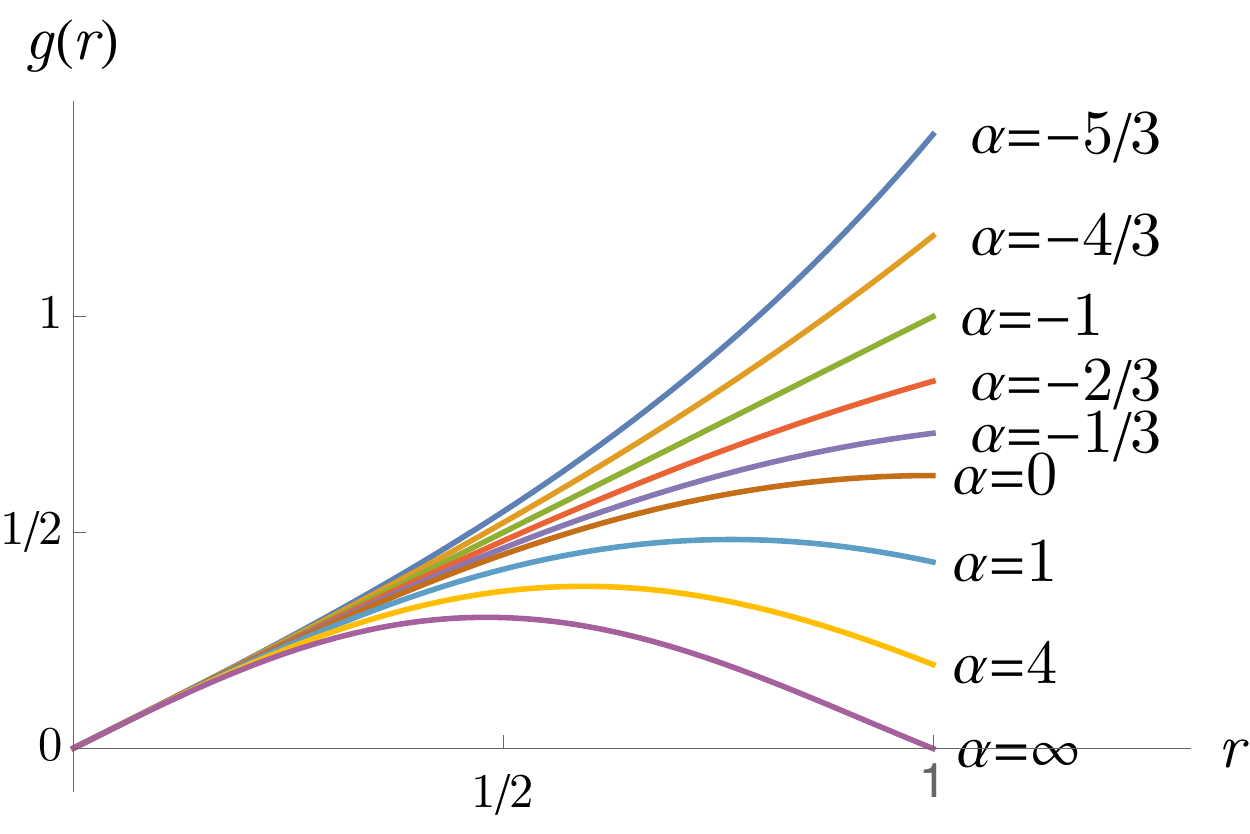}
\caption{\label{Robin_g_second}Plot of the radial part $g(r)$ of the second Robin eigenfunction of the unit disk, for various values of $\alpha$, normalized with $g^\prime(0)=1$. (When $\alpha=-1$ it is the straight line $g(r)=r$.) The eigenfunctions are $g(r) \cos \theta$ and $g(r) \sin \theta$; the eigenvalue has multiplicity $2$. This paper concentrates on $\alpha \leq 0$. }
\end{figure}
\begin{figure}
\includegraphics[scale=0.5]{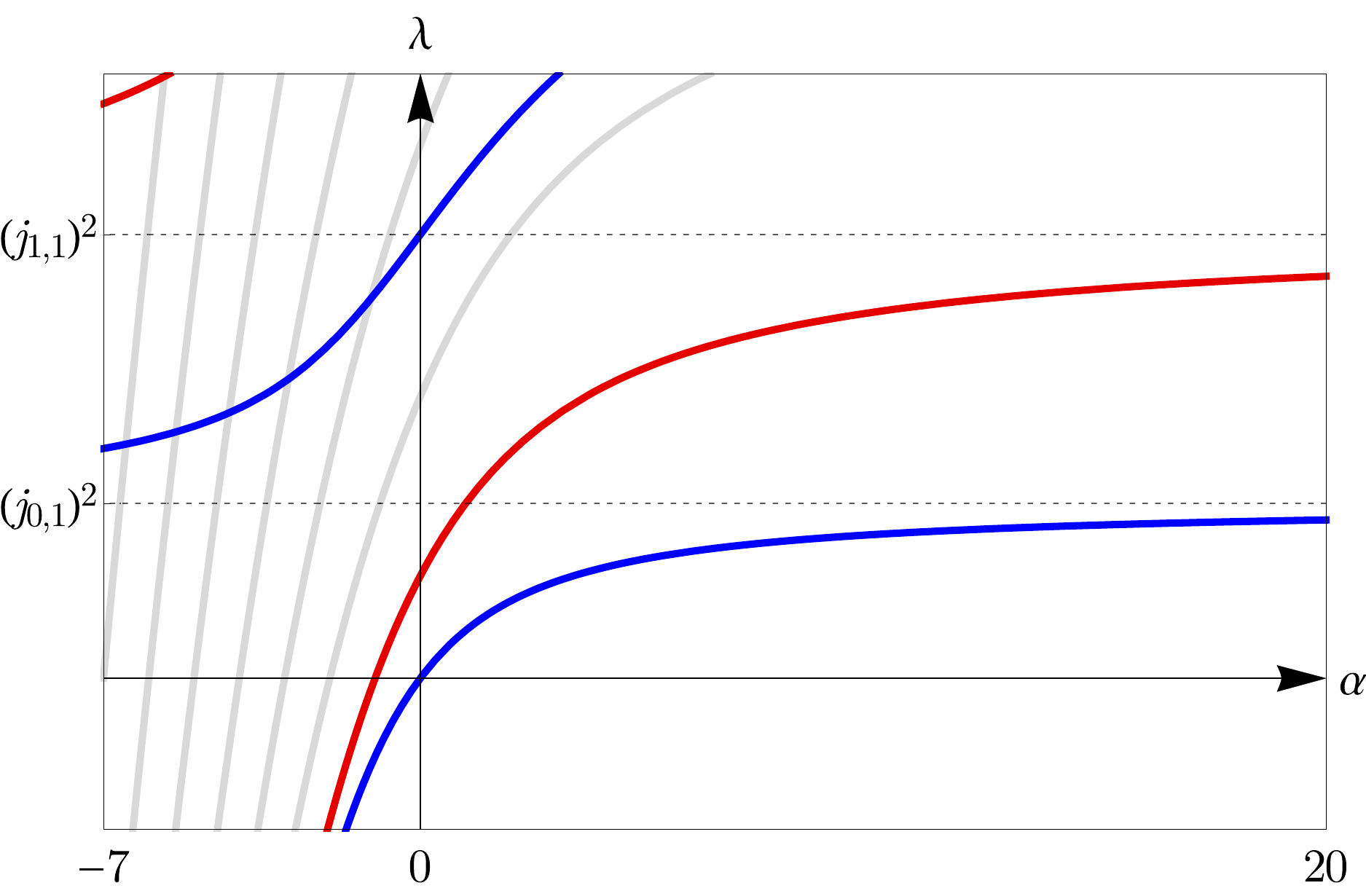}
\caption{\label{Robin12}Plot of the first two Robin eigenvalues of the unit disk: $\lambda_1(\D;\alpha)$ ($\kappa=0$), and
$\lambda_2(\D;\alpha)=\lambda_3(\D;\alpha)$ ($\kappa=1$), shown as the lowest curves in blue and red respectively. The next eigenvalues
for $\kappa=0,1$ are shown in the same colours, while eigencurves corresponding to higher values of $\kappa$ are shown in grey.
%The curves are shown for $\alpha \geq -7$, along with their horizontal asymptotes as $\alpha \to \infty$. 
To generate each curve one graphs $\alpha$ in terms of $\lambda$, by the relation $\alpha = -\sqrt{\lambda} G^\prime(\sqrt{\lambda})/G(\sqrt{\lambda})$, where $G=J_0$ for the first curve and $G=J_1$ for the second. (When $\lambda<0$, replace $\lambda$ by $|\lambda|$ and change the Bessel $J$-function to a Bessel $I$-function.)}
\end{figure}

We proceed now to analyze the spectrum, and establish the propositions. 

\subsection*{Zero eigenvalues}
Assume $\lambda=0$. Then \eqref{besseleq} simplifies to 
\[
 r^2 g''(r) + (n-1)rg'(r) - \kappa(\kappa+n-2) g(r) = 0 .
\]
This differential equation has solutions $r^\kappa$ and $r^{-(\kappa+n-2)}$, except that when $\kappa=0$ and $n=2$ the two solutions coincide, and the second solution should be replaced by $\log r$. We discard the second solution, in every case, since the eigenfunctions must have square integrable radial derivative, that is, $\int_0^1 g^\prime(r)^2 \, r^{n-1} dr < \infty$. Thus $g$ is given by the first solution:
\[
g(r) = r^\kappa , \qquad 0 \leq r \leq 1 .
\]
%Notice $\kappa \geq 1$ implies $g(0)=0$, while $\kappa \neq 1$ implies $g^\prime(0)=0$, and $\kappa = 1$ implies $g^\prime(0)>0$.
The Robin boundary condition $g^\prime(1)+\alpha g(1)=0$ requires $\kappa + \alpha = 0$, or $\alpha = -\kappa$. 

We conclude that zero eigenvalues occur at parameter values $\alpha=-\kappa$ for integers $\kappa \geq 0$, with corresponding eigenfunctions $u(r,\theta)=r^\kappa T(\theta)$ where $T$ is a spherical harmonic of degree $\kappa$. 

\subsection*{Negative eigenvalues}
Assume $\lambda<0$. Letting
\[
g(r) = G \big( \sqrt{-\lambda} \,r \big) 
\]
in \eqref{besseleq}, we find $G$ satisfies a differential equation that is independent of $\lambda$, namely
\begin{equation*}\label{modifiedbesseleq2}
 G''(r) + \fr{n-1}{r}G'(r) - \left( 1 + \frac{\kappa(\kappa+n-2)}{r^2} \right) G(r) = 0 .
\end{equation*}
The equation has solution 
\[
 G(r) = r^{1-n/2}I_{n/2+\kappa-1}(r)
\]
where $I_{\nu}$ is the modified Bessel function of the first kind. (We discard the modified Bessel functions $K_{n/2+\kappa-1}$ of the second kind, since we need eigenfunctions whose derivatives are square integrable.) Note $G(r) = (\text{const.}) r^\kappa + O(r^{\kappa+2})$ as $r \to 0$. Thus $\kappa \geq 1$ implies $G(0)=0$, while $\kappa \neq 1$ implies $G^\prime(0)=0$, and $\kappa = 1$ implies $G^\prime(0)>0$.

The Robin boundary condition says $g'(1) + \alpha g(1) = 0$, which is equivalent to $\sqrt{-\lambda} G^\prime(\sqrt{-\lambda})+\alpha G(\sqrt{-\lambda}) = 0$. To analyze this  condition, we logarithmically differentiate $G$ to obtain 
\[
r \frac{G^\prime(r)}{G(r)} = 1 - \frac{n}{2}  + r \frac{I_{n/2+\kappa-1}^\prime(r)}{I_{n/2+\kappa-1}(r)} . 
\]
The right side of this equation is strictly increasing, by \autoref{modifiedbesselmonot}. 

The Robin boundary condition in the last paragraph is
\begin{equation}\label{modifiedrobinbesselcondition}
y\fr{I_{n/2+\kappa-1}^\prime(y)}{I_{n/2+\kappa-1}(y)} = \frac{n}{2} - 1 -  \alpha ,
\end{equation}
where we have written $y=\sqrt{-\lambda}$. As $y$ increases from $0$ to $\infty$, the expression on the left strictly increases from $n/2+\kappa-1$ to $\infty$, by \autoref{modifiedbesselmonot}, and so \eqref{modifiedrobinbesselcondition} determines a unique solution  $y_\kappa(\alpha)>0$, when $\alpha<-\kappa$. Clearly $y_\kappa(\alpha)$ is a strictly decreasing function of $\alpha<-\kappa$, and so the eigenvalue $\lambda_\kappa(\alpha)=-y_\kappa(\alpha)^2$ strictly increases from $-\infty$ to $0$ as $\alpha$ increases from $-\infty$ to $-\kappa$. 

We will show
\begin{equation*} \label{modifiedmonotkappa}
y_\kappa(\alpha) > y_{\kappa+1}(\alpha) \qquad \text{whenever\ } \alpha < -(\kappa+1) ,
\end{equation*}
so that $\lambda_\kappa(\alpha) < \lambda_{\kappa+1}(\alpha)$, meaning the negative eigenvalue branches increase monotonically with respect to $\kappa$ wherever they are defined. Indeed, from \eqref{modifiedrobinbesselcondition} and the strictly increasing dependence with respect to $\nu$ in \autoref{modifiedbesselmonot} we find
\[
\frac{n}{2} - 1 -  \alpha = y_\kappa(\alpha) \fr{I_{n/2+\kappa-1}^\prime(y_\kappa(\alpha))}{I_{n/2+\kappa-1}(y_\kappa(\alpha))} < y_\kappa(\alpha) \fr{I_{n/2+\kappa}^\prime(y_\kappa(\alpha))}{I_{n/2+\kappa}(y_\kappa(\alpha))} ,
\]
so that $y_\kappa(\alpha)$ is larger than the root $y_{\kappa+1}(\alpha)$.

Since the negative eigenvalues increase in value with $\kappa$, we conclude that the lowest eigenvalue comes from the branch with $\kappa=0$, that is, $\lambda_1(\alpha) = -y_0(\alpha)^2$ when $\alpha<0$. The eigenfunction is $g(r) = G\big( \sqrt{-\lambda_1(\B;\alpha)} \, r \big)$ where $G(r)=r^{1-n/2}I_{n/2-1}(r)$. The power series for the modified Bessel function shows that $g(0)>0,g^\prime(0)=0$, and $g^\prime(r)>0$ for $r>0$.

The next lowest negative eigenvalue is associated with $\kappa=1$, that is, $\lambda_2(\alpha) = -y_1(\alpha)^2$ when $\alpha<-1$. The radial part of the eigenfunction is $g(r) = G\big( \sqrt{-\lambda_2(\B;\alpha)} \, r \big)$ where $G(r)=r^{1-n/2}I_{n/2}(r)$. The power series for the modified Bessel function shows $g(0)=0,g^\prime(0)>0$ and $g^\prime(r)>0$ for $r>0$. 

\subsection*{Positive eigenvalues}
Assume $\lambda>0$. Letting
\[
g(r) = G \big( \sqrt{\lambda} \,r \big) 
\]
in \eqref{besseleq}, we find again that $G$ satisfies a differential equation independent of $\lambda$, 
\begin{equation*}\label{besseleq2}
 G''(r) + \fr{n-1}{r}G'(r) + \left( 1 - \frac{\kappa(\kappa+n-2)}{r^2} \right) G(r) = 0 .
\end{equation*}
The solution is  
\[
 G(r) = r^{1-n/2}J_{n/2+\kappa-1}(r)
\]
where $J_{\nu}$ is the Bessel function of the first kind. (We discard the Bessel functions $Y_{n/2+\kappa-1}$ of the second kind, since we need eigenfunctions whose derivatives are square integrable.) Note that $G$ is called by some authors an ultraspherical Bessel function. It satisfies $G(r) = (\text{const.}) r^\kappa + O(r^{\kappa+2})$ as $r \to 0$. Thus $\kappa \geq 1$ implies $G(0)=0$, while $\kappa \neq 1$ implies $G^\prime(0)=0$, and $\kappa = 1$ implies $G^\prime(0)>0$.

The Robin boundary condition says $g'(1) + \alpha g(1) = 0$, which is equivalent to $\sqrt{\lambda} G^\prime(\sqrt{\lambda})+\alpha G(\sqrt{\lambda}) = 0$. We investigate by logarithmically differentiating $G$ to find 
\[
r \frac{G^\prime(r)}{G(r)} = 1 - \frac{n}{2}  + r \frac{J_{n/2+\kappa-1}^\prime(r)}{J_{n/2+\kappa-1}(r)} . 
\]
The right side of this equation is strictly decreasing, by \autoref{besselmonot}. 

The Robin boundary condition in the last paragraph is
\begin{equation}\label{robinbesselcondition}
x\fr{J_{n/2+\kappa-1}^\prime(x)}{J_{n/2+\kappa-1}(x)} = \frac{n}{2} - 1 -  \alpha ,
\end{equation}
where we have written $x=\sqrt{\lambda}$. The expression on the left behaves qualitatively like a negative tangent function for positive values of $x$, decreasing initially from $n/2+\kappa-1$ to $-\infty$, and then from $\infty$ to $-\infty$ between successive zeros of the denominator, as \autoref{besselmonot} shows. Each branch of the left side of \eqref{robinbesselcondition} determines $\alpha$ as a strictly increasing function of $x=\sqrt{\lambda}$. Taking the inverse function determines a branch of $\sqrt{\lambda}$ as a function of $\alpha$. 

For each fixed $\kappa$, the lowest branch of $\sqrt{\lambda}$ is defined for $\alpha> -\kappa$ and decreases to $0$ as $\alpha$ decreases to $-\kappa$, and increases to $j_{n/2+\kappa-1,1}$ as $\alpha \to \infty$. Each higher branch ($m \geq 1$) is defined for all $\alpha \in \R$ and decreases to $j_{n/2+\kappa-1,m}$ as $\alpha \to -\infty$, and increases to $j_{n/2+\kappa-1,m+1}$ as $\alpha \to \infty$. We will use these branches to study the positive Robin eigenvalues of the unit ball. 

Write $x_\kappa(\alpha)$ for the lowest solution branch of \eqref{robinbesselcondition}, when $\alpha>-\kappa$, so that $0<x_\kappa(\alpha)<j_{n/2+\kappa-1,1}$. We show
\begin{equation*} \label{monotkappa}
x_\kappa(\alpha) < x_{\kappa+1}(\alpha) \qquad \text{whenever\ } \alpha > -\kappa ,
\end{equation*}
so that the lowest eigenvalue branches increase monotonically with $\kappa$ wherever they are defined. We may suppose $x_{\kappa+1}(\alpha)<j_{n/2+\kappa-1,1}$, since otherwise there is nothing to prove. From \eqref{robinbesselcondition} and the strictly increasing dependence with respect to $\nu$ in \autoref{besselmonot} we find
\[
\frac{n}{2} - 1 -  \alpha
= x_\kappa(\alpha)\fr{J_{n/2+\kappa-1}^\prime(x_\kappa(\alpha))}{J_{n/2+\kappa-1}(x_\kappa(\alpha))} 
< x_\kappa(\alpha)\fr{J_{n/2+\kappa}^\prime(x_\kappa(\alpha))}{J_{n/2+\kappa}(x_\kappa(\alpha))} 
\]
which means that $x_\kappa(\alpha)$ is smaller than the root $x_{\kappa+1}(\alpha)$.

We conclude that the lowest eigenvalue comes from the branch with $\kappa=0$, that is, $\lambda_1(\alpha) = x_0(\alpha)^2$ when $\alpha>0$. The eigenfunction is $g(r) = G\big( \sqrt{\lambda_1(\B;\alpha)} \, r \big)$ where $G(r)=r^{1-n/2}J_{n/2-1}(r)$. The power series for the Bessel function gives $g(0)>0$ and $g^\prime(0)=0$. Also,  $G^\prime(r) = - r^{1-n/2}J_{n/2}(r)$ by \cite[10.6.6]{DLMF}, and since the construction above ensures $x_0(\alpha)<j_{n/2-1,1}$, we deduce $g^\prime(r)<0$ for $r \in (0,1)$. 

Next we show the second eigenvalue comes from the branch with $\kappa=1$. For this we must show 
\[
x_1(\alpha) < x_0^1(\alpha) \quad \text{whenever $\alpha > -1$,}
\]
where we write $x_0^1(\alpha)$ for the first higher branch with $\kappa=0$, that is, the branch with $m=1$ that is defined for all $\alpha \in \R$ and satisfies $j_{n/2-1,1} < x_0^1(\alpha) < j_{n/2-1,2}$. When $-1 < \alpha < 1$, we have $n \geq 2 > 1+\alpha$ and so
\[
\frac{n}{2} - 1 -  \alpha > - \frac{n}{2} = j_{n/2-1,1} \fr{J_{n/2}^\prime(j_{n/2-1,1})}{J_{n/2}(j_{n/2-1,1})}
\]
by the recurrence relation \cite[10.6.2]{DLMF}
\[
x\frac{J_{\nu}^\prime(x)}{J_{\nu}(x)} =  x \frac{J_{\nu-1}(x)}{J_{\nu}(x)} -\nu .
\]
It follows that $x_1(\alpha)<j_{n/2-1,1}$, which by definition is smaller than $x_0^1(\alpha)$. Hence $x_1(\alpha) < x_0^1(\alpha)$. 

Now suppose $\alpha > 0$. The quantity on the left of \eqref{robinbesselcondition} can be rewritten as
\[
x\fr{J_{n/2-1}^\prime(x)}{J_{n/2-1}(x)}  = - x \frac{J_{n/2}(x)}{J_{n/2-1}(x)} + \frac{n}{2} - 1
\] 
by another recurrence relation \cite[10.6.2]{DLMF}, and so this quantity can equal $n/2-1-\alpha$ if and only if 
\[
x \frac{J_{n/2}(x)}{J_{n/2-1}(x)} = \alpha .
\]
Thus the choice $x=x_0^1(\alpha)$ must make the left side positive, since $\alpha>0$. The denominator is negative, because $x_0^1(\alpha)$ lies between the first and second zeros of $J_{n/2-1}$. Thus the numerator must be negative at $x=x_0^1(\alpha)$, which means $x_0^1(\alpha)>j_{n/2,1}$, and that is larger than $x_1(\alpha) < j_{n/2,1}$ by construction. 

This completes the proof that the second eigenvalue comes from $\kappa=1$, that is, $\lambda_2(\alpha) = x_1(\alpha)^2$ when $\alpha>-1$. 

We have shown when $\alpha > -1$ that the second eigenvalue of the unit ball has $\kappa=1$, and its eigenfunction has radial part $g(r) = G\big( \sqrt{\lambda_2(\B;\alpha)} \, r \big)$ where $G(r)=r^{1-n/2}J_{n/2}(r)$. The square root of the eigenvalue is less than $j_{n/2,1}$. Hence by \autoref{besselmonot}, $rg^\prime(r)/g(r)$ is strictly decreasing on $r \in (0,1)$. It equals $-\alpha$ when $r=1$, by the Robin boundary condition. Hence $r g^\prime(r) + \alpha g(r) > 0$ when $0<r<1$. If $\alpha \in (-1,0]$ then it follows that $g^\prime(r)>0$ when $0<r<1$.

\section{\bf Explicit eigenvalue bounds for the ball}

The second eigenvalue $\lambda_2$ of the ball, which provides our upper bound in \autoref{mainthm}, may be computed numerically for each $\alpha$ from equation \eqref{robinbesselcondition}. Or one may use that formula to compute the inverse function, that is, to compute $\alpha$ in terms of $\lambda_2$. To complement those approaches, we obtain in this section accurate and explicit estimates for the second eigenvalue by means of inequalities on the quotient functions $J_{\nu+1}/J_{\nu}$ and $I_{\nu+1}/I_\nu$. 

First we consider $\alpha \in [-1,0]$. We will concentrate on an upper bound for the second eigenvalue, because that is more relevant to our work, but it is possible to obtain a lower bound in a similar fashion.
\begin{proposition} \label{lambdaupperbound}
 The second eigenvalue of the unit ball satisfies the estimate
 \[
 0 \leq \lambda_{2}(\B;\alpha) \leq \fr{1}{2} (n+2)(n+4) \left( \sqrt{1+ 4\fr{1+\alpha}{n+4}} - 1 \right) , \qquad \alpha\in[-1,0] ,
 \]
 with equality on both sides when $\alpha=-1$. 
\end{proposition}
\begin{proof}
The recurrence relation \cite[10.6.2]{DLMF} and formula \eqref{robinbesselcondition} together give
\[
x \fr{J_{n/2+1}(x)}{J_{n/2}(x)} = - x \fr{J_{n/2}^\prime(x)}{J_{n/2}(x)} + \frac{n}{2} = 1+\alpha .
\]
Meanwhile, from \cite[formula (1.2)]{IS90} we have
 \[
 x \fr{J_{n/2+1}(x)}{J_{n/2}(x)}\geq \fr{x^2}{n+2}\left( 1+\fr{x^2}{(n+2)(n+4)}\right) ,
 \]
noting the formula is valid here since $x = \sqrt{\lambda_2(\B;\alpha)} < j_{n/2,1}$. Combining the two relations, we deduce 
\[
x^4 +(n+2)(n+4)x^2 -(1+\alpha)(n+2)^2(n+4) \leq 0 .
\]
Now the quadratic formula implies the desired bound on $\lambda_2(\B;\alpha)=x^2$.
\end{proof}

Next consider $\alpha<-1$, in which range the second eigenvalue is negative. 
\begin{proposition} \label{lambdalowerbound}
If $\alpha < -1$ then 
\[
-(\alpha+1)^2 + (n+2)(\alpha+1) \leq \lambda_2(\B;\alpha) < -(\alpha+1)^2 + n(\alpha+1) .
\]
\end{proposition}
The proof will yield a slightly stronger upper bound than the one stated. 
\begin{proof}
First we show $\lambda_2(\B;\alpha) < (n/2+1)(\alpha+1)$, which is weaker than the upper bound we will ultimately prove. A recurrence relation \cite[10.29.2]{DLMF} for the modified Bessel function with $\nu=n/2$ gives 
\[
y \frac{I_{n/2}^\prime(y)}{I_{n/2}(y)} - \frac{n}{2} = y\frac{I_{n/2+1}(y)}{I_{n/2}(y)}  ,
\]
while applying another formula from \cite[10.29.2]{DLMF} with $\nu=n/2+1$ gives
\[
\frac{(n/2)+1}{y} I_{n/2+1}(y) = - I_{n/2+1}^\prime(y) + I_{n/2}(y) < I_{n/2}(y) .
\]
Putting these formulas into the eigenvalue condition \eqref{modifiedrobinbesselcondition} with $\kappa=1$, we obtain
\[
-(\alpha+1) = y\frac{I_{n/2+1}(y)}{I_{n/2}(y)} < \frac{y^2}{(n/2)+1} .
\]
Substituting $y=\sqrt{-\lambda_2(\B;\alpha)}$ yields 
\begin{equation} \label{lambda2prelim}
\lambda_2(\B;\alpha) < (n/2+1)(\alpha+1) , \qquad \alpha \in (-\infty,-1),
\end{equation}
as claimed. 

Next, a lower bound by Amos \cite[formula (9)]{A74} gives
\[
-(\alpha+1) = y\frac{I_{n/2+1}(y)}{I_{n/2}(y)} \geq \frac{y^2}{(n/2+1) + \sqrt{y^2 + (n/2+1)^2}} .
\]
Rearranging, 
\[
-(\alpha+1) \sqrt{y^2 + (n/2+1)^2} \geq y^2 + (n/2+1) (\alpha+1) .
\]
The right side of this inequality is positive by \eqref{lambda2prelim}. Squaring both sides and canceling the common factor of $y^2$ yields the lower bound on $\lambda_2(\B;\alpha)$ in the proposition. 

An upper bound by Amos \cite[formula (11)]{A74} gives
\[
-(\alpha+1) = y\frac{I_{n/2+1}(y)}{I_{n/2}(y)} \leq \frac{y^2}{n/2+ \sqrt{y^2 + (n/2+2)^2}} .
\]
Rearranging, 
\[
-(\alpha+1) \sqrt{y^2 + (n/2+2)^2} \leq y^2 + (n/2) (\alpha+1) .
\]
Squaring both sides and solving the resulting quadratic inequality for the negative quantity $\lambda_2(\B;\alpha) = -y^2$ yields
\[
\lambda_2(\B;\alpha) \leq \frac{1}{2} \left( -(\alpha+1)^2 + n(\alpha+1) - \sqrt{\big((\alpha+1)^2 - n(\alpha+1)\big)^2 + 8(n+2)(\alpha+1)^2} \right) .
\]
By discarding the second term in the square root we obtain the upper bound on $\lambda_2(\B;\alpha)$ in the proposition. 
\end{proof}

\section{\bf Proof of \autoref{mainthm} when $-R^{-1} \leq \alpha \leq 0$, and proof of \autoref{brockweinberger}}
\label{higherdimproof1}

First rescale the domain so that $\Omega$ has the same volume as the unit ball $\B$, using the scaling relation
\[
\lambda(\Omega;\alpha) = t^{-2} \, \lambda(t^{-1} \Omega ; t\alpha) , \qquad t>0 ,
\]
with the particular choice $t=R$. 

\smallskip Step 1. After this rescaling, we have $R=1,B=\B$, and the task is to prove
\[
\lambda_2(\Omega;\alpha) \leq \lambda_2(\B;\alpha) , \qquad \alpha \in [-1,0] .
\]
The restriction $\alpha \geq -1$ ensures by \autoref{basic2} that $\lambda_2(\B;\alpha) \geq 0$. Thus we may assume $\lambda_2(\Omega;\alpha) \geq 0$, since otherwise there is nothing to prove. 

\smallskip Step 2. To adapt Weinberger's method from the Neumann case \cite{W56}, we define trial functions by 
\[
v_{i+1}(x) = g(r)\frac{x_i}{r} , \qquad i=1,\dots,n, 
\]
where $g(r)$ equals the radial part of the second Robin eigenfunction of the unit ball. We constructed $g$ in \autoref{basic2} for $r \in [0,1]$, and now extend it outside the ball by 
\[
g(r) = g(1)e^{-\alpha(r-1)} , \qquad r > 1 .
\]
The slopes of $g$ from the left and right agree at $r=1$ because the Robin boundary condition on the unit sphere and the preceding formula for $g$ outside the sphere give 
\[
g^\prime(1-) = - \alpha g(1) = g^\prime(1+) .
\] 
Properties of $g$ we need from \autoref{basic2} are that $g(r)$ is nonnegative and strictly increasing for $r \in [0,1]$, with $g(0)=0, g^\prime(0)>0$. Note $g(r)$ is increasing for $r \geq 1$, and so $\int_0^\infty g(r) \, dr = \infty$. Clearly $v_{i+1}$ is $C^1$-smooth on $\Rn$. 

The center of mass result in \autoref{centerofmass} guarantees the domain $\Omega$ can be translated to make the following orthogonality conditions hold:  
\[
\int_\Omega v_{i+1} v_1 \, dx = 0 , \qquad i=1,\dots,n,
\]
where $v_1 \geq 0$ is the eigenfunction for the first eigenvalue $\lambda_1(\Omega;\alpha)$. Each $v_{i+1}$ is therefore a valid trial function for the second eigenvalue $\lambda_2(\Omega;\alpha)$. 

\smallskip Step 3. The Rayleigh characterization of the second eigenvalue implies
$\lambda_2(\Omega;\alpha) \leq Q[v_{i+1}]$, and so 
\[
\lambda_2(\Omega;\alpha) \int_\Omega v_{i+1}^2 \, dx 
\leq \int_\Omega |\nabla v_{i+1}|^2 \, dx + \alpha \int_{\partial \Omega} v_{i+1}^2 \, dS , \qquad i=1,\dots,n,
\]
with equality when $\Omega = \B$. Substituting the definition $v_{i+1} = g(r) x_i/r$, we obtain
\[
\begin{split}
& \lambda_2(\Omega;\alpha) \int_\Omega g(r)^2 x_i^2/r^2 \, dx \\
\leq & \int_\Omega \left( g^\prime(r)^2 x_i^2/r^2 + r^{-2} g(r)^2 (1-x_i^2/r^2) \right) dx + \alpha \int_{\partial \Omega} g(r)^2 x_i^2/r^2 \, dS .
\end{split}
\]
Summing over $i$ gives
\[
\lambda_2(\Omega;\alpha) \int_\Omega g(r)^2 \, dx \\
\leq \int_\Omega \left( g^\prime(r)^2 + (n-1) r^{-2} g(r)^2 \right) dx + \alpha \int_{\partial \Omega} g(r)^2 \, dS .
\]
Next we estimate the boundary integral with a domain integral: from formula \eqref{boundaryint} in \autoref{boundary}, and the fact that $\alpha \leq 0$, we deduce 
\begin{equation} \label{mainformulavol}
\lambda_2(\Omega;\alpha) \int_\Omega g(r)^2 \, dx 
\leq \int_\Omega h(r) \, dx
\end{equation}
where 
\begin{equation} \label{hdef}
h(r) = g^\prime(r)^2 + (n-1) r^{-2} g(r)^2 + 2\alpha g(r)g^\prime(r) + \alpha \frac{n-1}{r} g(r)^2 .
\end{equation}
Equality holds in \eqref{mainformulavol} if $\Omega=\B$. Note $h$ is continuous, since we constructed $g^\prime$ to be continuous even at $r=1$.

\smallskip Step 4. To continue adapting Weinberger's method, we show the integrands are monotonic: 
\begin{lemma}\label{monotonicity}
$g(r)^2$ is increasing and $h(r)$ is strictly decreasing, for $0<r<\infty$. 
\end{lemma}
\begin{proof}
By construction $g(r)$ is strictly increasing and positive for $0 < r \leq 1$, and is strictly increasing for $r \geq 1$ except when $\alpha=0$ (in which case $g(r)$ is constant for $r \geq 1$). The same properties hold for $g(r)^2$. 

For $h(r)$, when $0<r<1$ we differentiate the definition \eqref{hdef} to obtain
\begin{align*}
h^\prime(r) 
& = - 2\lambda_2(\B;\alpha) g(r) \big( g^\prime(r) + \alpha g(r) \big) - \frac{2(n-1)}{r} \left( g^\prime(r) - \frac{g(r)}{r} \right)^{\! 2} \\
& \qquad \qquad + \alpha (n-1) \frac{g(r)^2}{r^2} + 2\alpha g^\prime(r)^2 ,
\end{align*}
where we eliminated $g^{\prime\prime}(r)$ from the formula with the help of the differential equation \eqref{besseleq} (taking $\kappa=1$ there). The formula for $h^\prime(r)$ contains four terms. The second term is certainly less than or equal to zero, and so are the third and fourth terms since $\alpha \leq 0$. The first term is less than or equal to zero since $\lambda_2(\B;\alpha) \geq 0, g(r)>0$, and 
\[
g^\prime(r) + \alpha g(r) > r g^\prime(r) + \alpha g(r)  \geq 0
\]
by \autoref{basic2}. Thus $h^\prime(r) \leq 0$. In fact, when $\alpha=0$ the first term in $h^\prime(r)$ is negative, and when $-1 \leq \alpha < 0$ the fourth term is negative. Thus $h^\prime(r)<0$ when $0<r<1$. 

Now consider $r \geq 1$. Substituting $g(r)=g(1)e^{-\alpha(r-1)}$ into the definition \eqref{hdef} for $h(r)$ gives
\[
h(r) = g(1)^2 \left( -\alpha^2 + (n-1) \frac{1+\alpha r}{r^2} \right) e^{-2\alpha(r-1)} ,
\]
and so 
\[
h^\prime(r) = g(1)^2 \left( 2\alpha^3 - \frac{2(n-1)}{r^3} \Big[ (1+\alpha r)^2 - \alpha r/2 \Big] \right) e^{-2\alpha(r-1)} ,
\]
which is negative since $\alpha \leq 0$. 
\end{proof}

\smallskip Step 5. Since $g^2$ is increasing and $h$ is strictly decreasing by \autoref{monotonicity}, the mass transplantation in \autoref{transplantation} implies that
\begin{align} 
\int_\B g(r)^2 \, dx \leq \int_\Omega g(r)^2 \, dx , \label{mass1} \\
\int_\Omega h(r) \, dx \leq \int_\B h(r) \, dx . \label{mass2} 
\end{align}
In inequality \eqref{mass2}, equality holds if and only if $\Omega = \B$. 

Inserting the relations \eqref{mass1} and \eqref{mass2} into inequality \eqref{mainformulavol} gives 
\begin{equation} \label{usingmonot}
\lambda_2(\Omega;\alpha) \int_\B g(r)^2 \, dx 
\leq \int_\B h(r) \, dx ,
\end{equation}
where on the left side of the inequality we used the supposition $\lambda_2(\Omega;\alpha) \geq 0$. Our derivation shows that equality holds if and only if $\Omega = \B$. Hence $\lambda_2(\Omega;\alpha) \leq \lambda_2(\B;\alpha)$, with equality if and only if $\Omega = \B$. 

\smallskip
\noindent \emph{Remark.} The domain $\Omega$ in the proof above is obtained from the original domain by rescaling and translation. Thus the equality statement for the original domain is simply that it is a ball, not necessarily centered at the origin. 

\subsection*{Proof of \autoref{brockweinberger}}
(a) When $\alpha=0$, the Robin eigenvalue problem becomes the Neumann problem:
\[
\begin{split}
- \Delta u & = \mu u \ \quad \text{in $\Omega$,} \\
\frac{\partial u}{\partial\nu} & = 0 \qquad \text{on $\partial \Omega$.} 
\end{split}
\]
The Neumann spectrum is traditionally indexed from $j=0$ whereas we index the Robin spectrum from $j=1$, and so $0=\mu_0 < \mu_1 \leq \mu_2 \leq \dots$ with $\mu_1(\Omega)=\lambda_2(\Omega;0)$. Thus \autoref{mainthm} with $\alpha=0$ gives Weinberger's upper bound \cite{W56} on the first nontrivial  Neumann eigenvalue: $\mu_1(\Omega) \leq \mu_1(B)$ with equality if and only if $\Omega$ is a ball. 

(b) Next, the Steklov spectrum of the Laplacian is denoted $0=\sigma_0 < \sigma_1 \leq \sigma_2 \leq \dots$ where the eigenvalue problem is
\[
\begin{split}
\Delta u & = 0 \ \ \quad \text{in $\Omega$,} \\
\frac{\partial u}{\partial\nu} & = \sigma u \quad \text{on $\partial \Omega$.} 
\end{split}
\]
Thus $\sigma$ belongs to the Steklov spectrum exactly when $0$ belongs to the Robin spectrum with $\alpha=-\sigma$. 

To prove the claim $\sigma_1(\Omega) \leq \sigma_1(B)$ in the corollary, again rescale $\Omega$ so that it has the same volume as the unit ball. Choosing $\alpha=-1$ in \autoref{mainthm} implies
\[
\lambda_2\big( \Omega;-1 \big) \leq \lambda_2\big( \B;-1 \big) =0 .
\]
Choosing instead $\alpha=0$ gives $\lambda_2(\Omega;0)=\mu_1(\Omega) > 0$. Since Robin eigenvalues vary continuously with $\alpha$, some number $\widetilde{\alpha} \in [-1,0)$ must exist for which $\lambda_2(\Omega;\widetilde{\alpha}) = 0$. We choose $\widetilde{\alpha}$ to be the greatest such number, so that $\lambda_2(\Omega;\alpha) > 0$ for all $\alpha > \widetilde{\alpha}$. Then $-\widetilde{\alpha}$ belongs to the Steklov spectrum, and is the smallest positive Steklov eigenvalue. 

Hence $\sigma_1(\Omega) = -\widetilde{\alpha}\leq 1 = \sigma_1(\B)$, as we needed to show. If equality holds then $\widetilde{\alpha}=-1$, and so $\lambda_2(\Omega;-1) =  0 = \lambda_2(\B;-1)$. The equality statement in \autoref{mainthm} then implies $\Omega$ is a ball.

\section{\bf Proof of \autoref{mainthm} when $-\frac{n+1}{n} R^{-1} \leq \alpha < -R^{-1}$}
\label{higherdimproof-adapted}

By rescaling we may take $R=1$ so that $B$ is the unit ball $\B$, as previously. When $\alpha<-1$ the second eigenvalue of the unit ball is negative, and so the derivation of inequality \eqref{usingmonot} from \eqref{mass1} breaks down. So we modify the proof of \autoref{mainthm}. By subtracting $\lambda_2(\B;\alpha) \int_\Omega g(r)^2 \, dx$ from both sides of \eqref{mainformulavol} we find
\begin{equation} \label{mainformulavol-adapted}
\left( \lambda_2(\Omega;\alpha) - \lambda_2(\B;\alpha) \right) \int_\Omega g(r)^2 \, dx 
\leq \int_\Omega \widetilde{h}(r) \, dx
\end{equation}
where 
\begin{align} 
\widetilde{h}(r) 
& = h(r) - \lambda_2(\B;\alpha)g(r)^2 \notag \\
& = g^\prime(r)^2 + (n-1) r^{-2} g(r)^2 + 2\alpha g(r)g^\prime(r) + \alpha \frac{n-1}{r} g(r)^2 - \lambda_2(\B;\alpha)g(r)^2 . \label{hdef-adapted}
\end{align}
Note $\widetilde{h}$ is continuous, since $g$ and $g^\prime$ are continuous. We now modify Steps 4 and 5. 

\smallskip Step $\widetilde{4}$. The modified integrand is monotonic: 
\begin{lemma}\label{monotonicity-adapted}
If $\alpha \in [- \frac{n+1}{n},-1)$ then $\widetilde{h}(r)$ is strictly decreasing, for $0<r<\infty$. 
\end{lemma}
\begin{proof}
\autoref{lambdalowerbound} implies that 
\begin{align}
\lambda_2(\B;\alpha) + \alpha^2
& \geq -(\alpha+1)^2 + (n+2)(\alpha+1) + \alpha^2 \notag \\
& = n(\alpha+1) + 1 \notag \\
& \geq 0 \label{usefullower}
\end{align}
because $\alpha \geq -(n+1)/n$. We will use this fact below, when treating $\widetilde{h}(r)$.

First consider the range $0 < r < 1$. Differentiating definition \eqref{hdef-adapted} gives
\begin{align*}
\widetilde{h}^\prime(r) 
& = 2 \left( a g^\prime(r)^2 + b g^\prime(r) g(r)/r+c \big(g(r)/r\big)^2 \right)
\end{align*}
where we eliminated $g^{\prime\prime}(r)$ from the formulas with the help of the differential equation \eqref{besseleq} (taking $\kappa=1$ there), and the coefficients are 
\begin{align*}
a & = \alpha - \frac{n-1}{r} , \\
b & = - 2 \lambda r + 2 \frac{n-1}{r} , \\
c & = -\alpha \lambda r^2 + \frac{1}{2}(n-1) \alpha - \frac{n-1}{r} ,
\end{align*}
with $\lambda=\lambda_2(\B;\alpha)$.

The coefficients $a$ and $c$ are negative. We will show the discriminant $b^2-4ac$ is also negative, so that $\widetilde{h}^\prime(r)<0$ as desired. Fix the parameters $n \geq 2, 0 < r < 1, -\frac{n+1}{n} \leq \alpha < -1$. We will prove the discriminant is negative for all values $\lambda \in [-\alpha^2, 0]$, which suffices for our purposes since $\lambda_2(\B;\alpha)$ lies in that interval by inequality \eqref{usefullower}, when $-\frac{n+1}{n} \leq \alpha < -1$. 

The discriminant is a quadratic function of $\lambda$: let
\begin{align*}
q(\lambda) 
& = r^2 (b^2 - 4ac) \\
& = 2 r \left( (n-1) \alpha (n+1 - r \alpha) - 2 r \big(2(n-1) + r \alpha (n-1 - r \alpha) \big) \lambda + 2 r^3 \lambda^2 \right) ,
\end{align*}
where we multiplied the discriminant by $r^2$ for convenience. Clearly $q(\lambda)$ is convex, since $q^{\prime \prime}(\lambda)=8r^4>0$. In order to prove $q$ is negative we need only show negativity at the endpoints, that is, at $\lambda=0$ and $\lambda=-\alpha^2$. This is easily done, since 
\[
q(0) = 2 r \alpha (n-1) (n+1 - r \alpha) < 0
\]
and 
\[
q \left( -\alpha^2 \right) = \frac{1}{2} r \alpha (n-1) \left( 8 (r\alpha + 3/4)^2 + 4n - 1/2 \right) < 0 .
\]

It remains to consider the range $r \geq 1$. Substituting $g(r)=g(1)e^{-\alpha(r-1)}$ into the definition \eqref{hdef-adapted} gives
\[
\widetilde{h}(r) = g(1)^2 \left( (n-1) \frac{1+\alpha r}{r^2} -\alpha^2  - \lambda_2(\B;\alpha) \right) e^{-2\alpha(r-1)} ,
\]
and so 
\[
\widetilde{h}^\prime(r) = g(1)^2 \left( - \frac{2(n-1)}{r^3} (1+\alpha r)^2 + \frac{n-1}{r^2} \alpha + 2\alpha(\alpha^2 + \lambda_2(\B;\alpha)) \right) e^{-2\alpha(r-1)} .
\]
Recalling that $\lambda_2(\B;\alpha) + \alpha^2 \geq 0$ by \eqref{usefullower}, we conclude $\widetilde{h}^\prime(r) < 0$ when $r>1$, and so $\widetilde{h}(r)$ is strictly decreasing for $r \geq 1$. 
\end{proof}

\smallskip Step $\widetilde{5}$. Since $\widetilde{h}$ is strictly decreasing by \autoref{monotonicity-adapted}, the mass transplantation in \autoref{transplantation} implies that
\[
\int_\Omega \widetilde{h}(r) \, dx \leq \int_\B \widetilde{h}(r) \, dx , 
\]
with equality if and only if $\Omega = \B$. Inserting this last inequality into \eqref{mainformulavol-adapted}, we deduce  
\[
\left( \lambda_2(\Omega;\alpha) - \lambda_2(\B;\alpha) \right) \int_\Omega g(r)^2 \, dx 
\leq \int_\B \widetilde{h}(r) \, dx 
\]
with equality if and only if $\Omega = \B$. The left side equals $0$ when $\Omega = \B$ and so the right side must equal $0$. (Alternatively, the right side can be evaluated to equal $0$ by using the definition of $\widetilde{h}$ and the differential equation satisfied by $g$.) Thus 
\[
\lambda_2(\Omega;\alpha) - \lambda_2(\B;\alpha)  \leq 0 
\]
with equality if and only if $\Omega = \B$, which completes the proof of the theorem.

\section*{Acknowledgments}
This research was supported by a grant from the Simons Foundation (\#429422 to Richard Laugesen), by travel support for Laugesen from the American Institute of Mathematics to the workshop on \emph{Steklov Eigenproblems} (April--May 2018), and by the Funda\c c\~{a}o para a Ci\^{e}ncia e a Tecnologia (Portugal) through project PTDC/MAT-CAL/4334/2014
(Pedro Freitas). The research advanced considerably during the workshop on \emph{Eigenvalues and Inequalities} at Institut Mittag--Leffler (May 2018), organized by Rafael Benguria, Hynek Kovarik and Timo Weidl. 

\appendix

\section{\bf Bessel function facts}

This appendix establishes monotonicity properties of Bessel and modified Bessel functions that were used in \autoref{preliminaries} to analyze the Robin spectrum of the ball.   
\begin{lemma} \label{modifiedbesselmonot}
Let $\nu \geq 0$. The function $rI_\nu^\prime(r)/I_\nu(r)$ is strictly increasing from $\nu$ to $\infty$ as $r$ increases from $0$ to $\infty$. Further, for each positive $r$ the expression $rI_\nu^\prime(r)/I_\nu(r)$ is strictly increasing as a function of $\nu \geq 0$.
\end{lemma}
\begin{proof}
The infinite product expansion of the modified Bessel function (found by using the relation $I_\nu(r)=i^{-\nu}J_\nu(ir)$ and the product for $J_\nu$ in \cite[10.21.15]{DLMF}) gives
\begin{align}
r \frac{I_\nu^\prime(r)}{I_\nu(r)}
& = r \frac{d\ }{dr} \log I_\nu(r) \notag \\
& = \nu + 2 \sum_{m=1}^\infty \frac{r^2}{r^2 + j_{\nu,m}^2} \label{modifiedlogderiv1} .
\end{align}
Each term of the sum is strictly increasing as a function of $r$, by \eqref{modifiedlogderiv1}, and the sum clearly tends to $\infty$ as $r \to \infty$.

For monotonicity with respect to $\nu$, let $\delta>0$ and $0 \leq \nu < \nu+\delta$. We want to show 
\[
r \frac{I_\nu^\prime(r)}{I_\nu(r)} < r \frac{I_{\nu+\delta}^\prime(r)}{I_{\nu+\delta}(r)} 
\]
for each $r>0$, which is equivalent to 
\[
\left( \log \frac{I_{\nu+\delta}(r)}{I_\nu(r)} \right)^{\! \prime} > 0 .
\]
This inequality holds for small $r>0$, since the leading order term in the power series for $I_\nu(r)$ is $(r/2)^\nu$ (see \cite[10.25.2]{DLMF}), and similarly for $I_{\nu+\delta}(r)$. 

To show the inequality holds for all $r>0$, it suffices to establish the following implication:
\[
\left( \log \frac{I_{\nu+\delta}(r)}{I_\nu(r)} \right)^{\! \prime} = 0 \quad \Longrightarrow \quad \left( \log \frac{I_{\nu+\delta}(r)}{I_\nu(r)} \right)^{\! \prime \prime} > 0 .
\]
Expanded out, the implication reads
\[
\frac{I_\nu^\prime(r)}{I_\nu(r)} = \frac{I_{\nu+\delta}^\prime(r)}{I_{\nu+\delta}(r)} 
\quad \Longrightarrow \quad 
\frac{I_\nu^{\prime \prime}(r)}{I_\nu(r)} - \left( \frac{I_\nu^\prime(r)}{I_\nu(r)} \right)^{\! 2}< \frac{I_{\nu+\delta}^{\prime \prime}(r)}{I_{\nu+\delta}(r)} - \left( \frac{I_{\nu+\delta}^\prime(r)}{I_{\nu+\delta}(r)} \right)^{\! 2} .
\]
We may substitute for the second derivatives using the modified Bessel equation
\[
r^2 I_\nu^{\prime \prime}(r) + r I_\nu^\prime(r) - (r^2+\nu^2) I_\nu(r) = 0 ,
\]
thereby reducing the desired conclusion to
\[
\frac{\nu^2}{r^2} - \frac{1}{r} \frac{I_\nu^\prime(r)}{I_\nu(r)} - \left( \frac{I_\nu^\prime(r)}{I_\nu(r)} \right)^{\! 2} < \frac{(\nu+\delta)^2}{r^2} - \frac{1}{r} \frac{I_{\nu+\delta}^\prime(r)}{I_{\nu+\delta}(r)} - \left( \frac{I_{\nu+\delta}^\prime(r)}{I_{\nu+\delta}(r)} \right)^{\! 2} .
\]
Applying the hypothesis reduces the inequality to $\nu^2 < (\nu+\delta)^2$, which is true.
\end{proof}

\begin{lemma} \label{besselmonot}
Let $\nu \geq 0$. The function $rJ_\nu^\prime(r)/J_\nu(r)$ is strictly decreasing from $\nu$ to $-\infty$ on the interval $(0,j_{\nu,1})$, and strictly decreasing from $\infty$ to $-\infty$ on each interval $(j_{\nu,m},j_{\nu,m+1})$ for $m \geq 1$, that is, between successive zeros of the denominator. 

Further, for each positive $r$, the expression $rJ_\nu^\prime(r)/J_\nu(r)$ is strictly increasing as a function of $\nu \geq 0$ on each interval of $\nu$-values on which $r$ is not a zero of $J_\nu$.
\end{lemma}
\begin{proof}
The infinite product expansion of the Bessel function \cite[10.21.15]{DLMF} gives
\begin{align}
r \frac{J_\nu^\prime(r)}{J_\nu(r)}
& = r \frac{d\ }{dr} \log J_\nu(r) \notag \\
& = \nu + 2 \sum_{m=1}^\infty \frac{r^2}{r^2 - j_{\nu,m}^2} \label{logderiv1} \\
& = \nu + 2 \sum_{m=1}^\infty \left( \frac{j_{\nu,m}^2}{r^2 - j_{\nu,m}^2} + 1 \right) . \label{logderiv2}
\end{align}
Each term of the sum is strictly decreasing as a function of $r$ wherever it is defined, by \eqref{logderiv2}, and the sum clearly tends to $-\infty$ as $r \to j_{\nu,m}$ from the left, and tends to $\infty$ as $r \to j_{\nu,m}$ from the right.. 

Further, each zero $j_{\nu,m}$ is strictly increasing as a function of $\nu$, by \cite[10.21(iv)]{DLMF}, and so the expression in \eqref{logderiv1} is strictly increasing as a function of $\nu$, when $r$ is fixed, provided $\nu$ avoids values at which $r=j_{\nu,m}$ for some $m$.

\emph{Aside.} Monotonicity with respect to $\nu$ can also be proved by studying the second derivative of $\log |J_{\nu+1}/J_\nu|$, like we did for modified Bessel functions in the proof of \autoref{modifiedbesselmonot}.
\end{proof}

\end{document}